\documentclass[11pt]{amsart}
\usepackage{amsfonts,amssymb}

\newcommand{\Ra}{\Rightarrow}

\newcommand{\IZ}{\mathbb{Z}}
\newcommand{\IN}{\mathbb{N}}

\newcommand{\pack}{\mathrm{pack}}
\newcommand{\Pack}{\mathrm{Pack}}
\newcommand{\I}{\mathcal I}
\newcommand{\J}{\mathcal J}
\newcommand{\F}{\mathcal F}
\newcommand{\N}{\mathcal N}
\newcommand{\M}{\mathcal M}
\newcommand{\D}{\mathcal D}
\newcommand{\CS}{\mathcal S}
\newcommand{\e}{\varepsilon}

\newcommand{\w}{\omega}

\newtheorem{proposition}{Proposition}
\newtheorem{theorem}{Theorem}

\newtheorem{corollary}{Corollary}
\theoremstyle{definition}
\newtheorem{example}{Example}
\newtheorem{definition}{Definition}
\newtheorem{problem}{Problem}

\title{Completeness of translation-invariant ideals in groups}
\author{T.~Banakh, N.~Lyaskovska}

\address{Department of Mathematics, Ivan Franko National University of Lviv, Ukraine, and Instytut Matematyki, Uniwersytet Humanistyczno-Przyrodniczy Jana Kochanowskiego w Kielcach, Poland}
\email{tbanakh@yahoo.com}
\address{Department of Mathematics, Ivan Franko National University of Lviv, Ukraine}
\email{lyaskovska@yahoo.com}

\begin{document}


\begin{abstract}
We introduce and study varions notions of completeness of translation-invariant ideals in groups. 
\end{abstract}

\maketitle

In this paper we introduce and study various notions of completeness of translation-invariant ideals in groups. Those completeness notions have  topological, measure-theoretic, or packing nature.

\section{The ideal of small subsets in a group} A proper family $\I\subsetneq\mathcal P(G)$ of subsets of a group $G$ is an {\em ideal} if $\I$ is closed under taking subsets and unions. An ideal $\I$ on $G$ is {\em translation-invariant} (briefly, {\em invariant}) if $xA\in\I$ for each $x\in G$ and $A\in\I$.

 Each group $G$ possesses the trivial ideal $\I_0=\{\emptyset\}$ and this is a unique invariant
ideal on a finite group. Each infinite group $G$ possesses the invariant
ideal $\F$ consisting of all finite subsets of $G$ and this is the smallest
invariant ideal containing a non-empty set.

A less trivial example of an invariant ideal is the ideal of small sets.
A subset $A$ of a group $G$ is called
\begin{itemize}
\item {\em large} if $FA=G$ for some finite set $F\subset G$;
\item {\em small} if for each large set $B\subset G$ the complement $B\setminus A$ is large.
\end{itemize}

It follows from the definition that the family $\mathcal S$ of small subsets is an invariant ideal in $G$.

Small sets admit a topological characterization.

\begin{theorem}\label{smallchar} For a subset $A$ of a group $G$ the following conditions are equivalent:
\begin{enumerate}
\item[1)] $A$ is small;
\item[2)] for any finite set $F\subset G$ the complement $G\setminus FA$ is large;
\item[3)] $A$ is nowhere dense with respect to some left-invariant totally bounded topology on $G$.
\end{enumerate}
\end{theorem}

A topology $\tau$ on a group $G$ is {\em totally bounded} if each  non-empty open set $U\in\tau$ is large.

\begin{proof} The equivalence $(1)\Leftrightarrow(2)$ was proved in
\cite{BM1}, \cite{BM2}.
\smallskip

$(2)\Ra(3)$ Assuming that $A\subset G$ satisfies (2), observe that the topology
$$\tau=\{\emptyset\}\cup\{U\subset G:\mbox{$U\supset G\setminus FA$ for some finite $F\subset G$}\}$$ on $G$ is left-invariant and totally bounded. The set $A$ is closed in $(G,\tau)$ and has empty interior. Otherwise we could find a finite subset $F\subset G$ with $G\setminus FA\subset A$. Then for the finite set $F_e=F\cup\{e\}$ we get $F_eA=FA\cup A=G$, which contradicts (2).
\smallskip

$(3)\Ra(2)$ Assume that a subset $A\subset G$ is nowhere dense in some totally bounded left-invariant topology $\tau$. Then for every finite $F\subset G$ the set $FA$ is nowhere dense in $(G,\tau)$ and hence the complement $G\setminus FA$ contains some open set $U\in\tau$. Since $\tau$ is totally bounded, $U$ is large in $G$. Consequently, there is a finite subset $B\subset G$ such that $G=BU\subset B(G\setminus FA)\subset G$, witnessing that (2) holds.
\end{proof}

By \cite{PSib} or \cite[1.3, 9.4]{P}, on any group $G$ there exists a totally bounded left-invariant topology $\tau$, which is Hausdorff and extremally disconnected. The latter means that the closure of each open subset of $(G,\tau)$ is open. This result of I.Protasov implies the following description of the ideal of small sets.

\begin{theorem} The ideal $\mathcal S$ of small subsets of any group $G$ coincides with the ideal of nowhere dense subsets of $G$ endowed with some Hausdorff extremally disconnected left-invariant topology $\tau_s$.
\end{theorem}

\begin{proof} According to \cite{PSib} or \cite[1.3, 9.4]{P}, the group $G$ admits a Hausdorff extremally disconnected totally bounded left-invariant topology $\tau$. Let $\tau_s$ be the topology on $G$ generated by the base $$\mathcal B=\{U\setminus A:U\in\tau,\;A\in\mathcal S\}.$$

Let us show that the topology $\tau_s$ is totally bounded. Given any basic set $U\setminus A\in\mathcal B$, use the total boundedness of the topology $\tau$ in order to find a finite subset $F\subset G$ such that $G=FU$. Since the set $A$ is small, for the finite set $F$ there is a finite set $E\subset G$ such that $E(G\setminus FA)=G$.
Now observe that $$EF(U\setminus A)\supset E(FU\setminus FA)=E(G\setminus FA)=G,$$ witnessing that the set $U\setminus A$ is large.
\smallskip

Applying Theorem~\ref{smallchar}, we conclude that the ideal $\M$ of nowhere dense subset of the totally bounded left-topological group $(G,\tau_s)$ lies in the ideal $\mathcal S$ of small sets.

In order to prove the reverse inclusion $\mathcal S\subset\M$, fix any small set $S\subset G$. It follows from the definition of the topology $\tau_s$ that $S$ is closed in $(G,\tau_s)$. We claim that $S$ is nowhere dense. Assuming the converse, we would find a non-empty basic set $U\setminus A\subset S$ with $U\in\tau$ and $A\in\mathcal S$. Then the set $U\subset A\cup S$ is small, which contradicts the total boundedness of the topology $\tau$. This completes the proof of the equality $\mathcal S=\M$.
\smallskip

Now we check that the topology $\tau_s$ is Hausdorff and extremally disconnected.
Since the topology $\tau$ is Hausdorff, so is the topology $\tau_s\supset\tau$.

In order to prove the extremal disconnectedness of the topology $\tau_s$, take any open  subset $W\in\tau_s$ and write it as the union $W=\bigcup_{i\in\I}(U_i\setminus A_i)$ of basic sets $U_i\setminus A_i\in\mathcal B$, $i\in\I$. Consider the open set $U=\bigcup_{i\in\I}U_i\in\tau$ and its closure $\overline{U}$ in $(G,\tau)$. The extremal disconnectedness of the topology $\tau$ implies that $\overline{U}\in\tau\subset\tau_s$. It remains to check that $W$ is dense in $\overline{U}$ in the topology $\tau_s$. In the opposite case, we would find a non-empty basic set $V\setminus A\in\mathcal B$ that meets $\overline{U}$ but is disjoint with $W$. The set $V$ is open in $\tau$ and meets the closure $\overline{U}$ of the open set $U\in\tau$. Consequently, $V\cap U\ne\emptyset$ and there is an index $i\in\I$ such that $V\cap U_i\ne \emptyset$. Then
$$(U_i\cap V\setminus A)\cap (U_i\setminus A_i)\subset (U_i\cap V\setminus A)\cap  W=\emptyset$$implies that $V\cap U_i\subset A\cup A_i$. But this is impossible as $V\cap U_i\in\tau$ is large while the set $A\cup A_i$ is small. This contradiction shows that the set $W$ is dense in $\overline{U}\in\tau_s$ in the topology $\tau_s$.
Consequently, the closure of the open set $W$ in the topology $\tau_s$ in open and the left-topological group $(G,\tau_s)$ is extremally disconnected.
\end{proof}

\begin{problem} Is the ideal $\mathcal S$ of small subsets of the group $\IZ$ equal to the ideal of nowhere dense subsets of $\IZ$ with respect to some {\em regular} totally bounded left-invariant topology $\tau$ on $\IZ$?
\end{problem}

In light of this problem, let us remark that each countable Abelian group $G$ contains  a small subset which is dense with respect to any totally bounded group topology on $G$, see \cite{BM2}.
\smallskip

\section{$\CS$-complete ideals}

In this section, given an invariant ideal $\I$ in a group $G$, we introduce an invariant ideal $\CS_\I$ called the {\em $\CS$-completion} of $\I$. For two subsets $A,B\subset G$ we write $A\subset_\I B$ if $A\setminus B\in\I$ and $A=_\I B$ if $A\subset_\I B$ and $B\subset_\I A$.

We define a subset $A$ of a group $G$ to be
\begin{itemize}
\item {\em $\I$-large} if $FA=_\I G$ for some finite subset $F\subset G$;
\item {\em $\I$-small} if for each $\I$-large subset $L\subset G$ the complement $L\setminus A$ is $\I$-large.
\end{itemize}

$\I$-Small subsets admit the following characterization.

\begin{proposition} A subset $A\subset G$ is $\I$-small if and only if for every finite subset $F\subset G$ the set $G\setminus FA$ is $\I$-large. Consequently, each small subset of $G$ is $\I$-small.
\end{proposition}

\begin{proof} The ``only if'' part follows immediately from the definition of an $\I$-small set. To prove the ``if'' part, assume that $G\setminus FA$ is $\I$-large for any finite subset $F\subset G$. Given an $\I$-large subset $L\subset G$ we need to show that $L\setminus A$ is $\I$-large.

Find a finite subset $F\subset G$ such that $FL=_\I G$. By our assumption, the set $G\setminus FA$ is $\I$-large and hence $E(G\setminus FA)=_\I G$ for some finite set $E\subset G$. Consider the finite set $EF$ and observe that
$$EF(L\setminus A)\supset E(FL\setminus FA)=_\I E(G\setminus FA)=_\I G$$  and hence $EF(L\setminus A)=_\I G$ witnessing that the set $L\setminus A$ is $\I$-large.
\end{proof}

It follows from the definition that the family $\CS_\I$ of all $\I$-small subsets is an invariant ideal that contains the ideal $\I$. This ideal $\CS_\I$ is called the {\em $\CS$-completion} of the ideal $\I$.
An invariant ideal $\I$ in a group $G$ is called {\em $\CS$-complete} if $\CS_\I=\I$.

It is clear that for the smallest ideal $\I_0=\{\emptyset\}$ its $\CS$-completion $\CS_{\I_0}$ coincides with the ideal $\CS$ of all small subsets of $G$.

\begin{theorem} For any invariant ideal $\I$ the ideal $\CS_\I$ is $\CS$-complete.  In particular, the ideal $\CS=\CS_{\I_0}$ is $\CS$-complete.
\end{theorem}

\begin{proof} The $\CS$-completeness of the ideal $\CS_\I$ will follow as soon as we check that each $\CS_\I$-small subset $A\subset G$ is $\I$-small.
Given a finite subset $F\subset G$, we need to show that the set $G\setminus FA$ is $\I$-large. Since $A$ is $\CS_\I$-small, $G\setminus FA$ is $\CS_\I$-large. So, there is a finite subset $E\subset G$ such that $G\setminus E(G\setminus FA)$ is $\I$-small and then $E(G\setminus FA)$ is $\I$-large and so is the set $G\setminus FA$.
\end{proof}

\begin{corollary} For any invariant ideal $\I\subset\CS$ we get $\CS_\I=\CS_{\CS}=\CS$.
\end{corollary}

The definition of the $\I$-small set implies that the $\CS$-completeness is preserved by arbitrary intersections.

\begin{proposition} For any $\CS$-complete ideals  $\I_\alpha$, $\alpha\in A$, in  a group $G$ the intersection $\I=\bigcap_{\alpha\in A}\I_\alpha$ is an $\CS$-complete ideal in $G$.
\end{proposition}

\section{The $\N$-completion of an invariant ideal in an amenable group}

It turns out that the $\CS$-completion $\CS_\I$ of an invariant ideal $\I$ in an amenable group $G$ contains another interesting ideal $\N_\I$ called the $\N$-completion of $\I$. 

Let us recall that a group $G$ is called {\em amenable} if $G$ has a {\em Banach measure}, which is a left-invariant probability measure $\mu:\mathcal P(G)\to[0,1]$ defined on the family of all subsets of $G$, see \cite{Pat}. By the F\o lner condition \cite[0.7]{Pat}, a group $G$ is amenable if and only if for every finite set $F\subset G$ and every $\e>0$ there is a finite set $E\subset G$ such that $|E\triangle xE|<\e|E|$ for all $x\in F$. Here $A\triangle B=(A\setminus B)\cup(B\setminus A)$ is the symmetric difference of two sets $A,B$. The class of amenable groups contains all
 locally finite and all abelian groups and is closed under many operations over groups, see \cite[0.16]{Pat}.

It is clear that for each Banach measure $\mu$ on an amenable group $G$ the family of $\mu$-null sets
$$\N_\mu=\{A\subset G:\mu(A)=0\} $$is an invariant ideal in $G$.

The following theorem suggested to the authors by I.V.Protasov shows that the ideals $\N_\mu$ form a cofinal subset in the family of all invariant ideals on an amenable group.

\begin{theorem}\label{Lem1} Each invariant ideal $\I$ in an amenable group $G$ lies in the ideal $\N_\mu$ for a suitable Banach measure $\mu$ on $G$.
\end{theorem}

\begin{proof} The Banach measure $\mu$ with $\N_\mu\supset\I$ will be constructed as a limit point of a net of probability measures $(\mu_d)_{d\in\D}$ 
indexed by elements of the directed set $\D=\F\times\IN\times\I$ endowed with the partial order $(F,n,A)\le (E,m,B)$ iff $F\subset E$, $n\le m$, and $A\subset B$. Here $\F$ is the family of finite subsets of the group $G$. 

To each triple $d=(F,n,A)\in\D$ we assign a probability measure $\mu_d$ on $G$ in the following manner. Using the F\o lner criterion of amenability \cite[0.7]{Pat}, find a finite subset $F_d\subset G$ such that 
$$\frac{|F_d\,\triangle\, xF_d|}{|F_d|}<\frac1n\quad\mbox{for all $x\in F$}.$$
Since $A$ belongs to an invariant ideal, we can find a point  $y_d\in G\setminus F_d^{-1} A$. Now consider the probability measure $\mu_d:\mathcal P(G)\to[0,1]$ defined by $$\mu_d(B)=\frac{|B\cap F_dy_d|}{|F_d|}\quad\mbox{for $B\subset G$}
$$and observe that $\mu_d(A)=0$.

Each measure $\mu_{d}$, $d\in\D$, being a function on the family $\mathcal P(G)$ of all subsets of the group $G$, is a point of the Tychonov cube $[0,1]^{\mathcal P(G)}$.

By the compactness of $[0,1]^{\mathcal P(G)}$,  the net of measures $(\mu_{d})_{d\in\D}$ has a limit point $\mu\in[0,1]^{\mathcal P(G)}$, see \cite[3.1.23]{En}. This is a function $\mu:\mathcal P(G)\to[0,1]$ such that for each  neighborhood $O(\mu)$ of $\mu$ in $[0,1]^{\mathcal P(G)}$ and every $d_0\in\D$ there is $d\ge d_0$ in $\D$ with $\mu_{d}\in O(\mu)$.

We claim that $\mu$ is a Banach measure on $G$ with $\I\subset\N_\mu$. We need to check the following conditions:
\begin{enumerate}
\item $\mu(G)=1$;
\item $\mu(A\cup B)=\mu(A)+\mu(B)$ for any disjoint sets $A,B\subset G$;
\item $\mu(xB)=\mu(B)$ for every $x\in G$ and $B\subset G$;
\item $\mu(B)=0$ for each $B\in\I$.
\end{enumerate}

1. Assuming that $\mu(G)<1$, consider the neighborhood
$O_1(\mu)=\{\eta\in [0,1]^{\mathcal P(G)}:\eta(G)<1\}$ and note
that in contains no measure $\mu_{d}$, $d\in\D$.
\smallskip

2. Assuming that $\mu(A\cup B)\ne\mu(A)+\mu(B)$ for some disjoint sets $A,B\subset G$, observe that $O_2(\mu)=\{\eta\in[0,1]^{\mathcal P(G)}:\eta(A\cup B)\ne\eta(A)+\eta(B)\}$ is an open neighborhood of $\mu$ in $[0,1]^{\mathcal P(G)}$ containing no measure $\mu_{d}$, $d\in\D$.
\smallskip

3. Assume that $\mu(xB)\ne \mu(B)$ for some $x\in G$ and
$B\subset G$. Find $n\in\IN$ such that $\frac3n<|\mu(xB)-\mu(B)|$ and consider the 
element $d_0=(\{x^{-1}\},n,\emptyset)\in\D$. 
  Since $\mu$ is a limit point of the net 
  $(\mu_d)_{d\in\D}$, the neighborhood
   $$O_3(\mu)=\Big\{\eta\in[0,1]^{\mathcal P(G)}\colon\max\{|\eta(A)-\mu(A)|,|\eta(xA)-\mu(xA)|\}<\frac1n\Big\}$$ of $\mu$
   contains the measure $\mu_d$ for some $d=(F,m,B)\in\D$ with $d\ge d_0$.

Since $\{x^{-1}\}\subset F$, the definition of the set $F_d$ guarantees that
$\frac{|F_d\triangle x^{-1}F_d|}{|F_d|}<\frac1m\le\frac1n$.
Then $$
\begin{aligned}
|\mu_d(B)-\mu_d(xB)|&=\frac{\big||B\cap F_dy_d|-|B\cap x^{-1}F_dy_d|\big|}{|F_d|}\le\\&\le \frac{|F_dy_d\,\triangle \, x^{-1}F_dy_d|}{|F_d|}=\frac{|F_d\,\triangle \, x^{-1}F_d|}{|F_d|}<\frac1m\le \frac1n.
\end{aligned}$$On the other hand, $\mu_d\in O(\mu)$ implies
$$|\mu_d(B)-\mu_d(xB)|\ge|\mu(B)-\mu(xB)|-|\mu(B)-\mu_d(B)|-|\mu(xB)-\mu_d(xB)|>\frac1n$$and this is a contradiction.
\smallskip

4. Take any $A_0\in\I$ and assume that $\mu(A_0)\ne 0$. Find $n\in\IN$ such that $\mu(A_0)>\frac 1n$ and consider the element $d_0=(\emptyset,n,A_0)\in\D$. Since $\mu$ is a limit of the net $(\mu_d)_{d\in\D}$, the neighborhood $O(\mu)=\{\eta\in [0,1]^{\mathcal P(G)}:\eta(A_0)>\frac1n\}$ contains the measure $\mu_d$ for some $d=(F,m,A)\ge d_0$ in $\D$. The choice of the point $y_d$ guarantees that $F_dy_d\cap A_0\subset F_dy_d\cap A=\emptyset$ and hence $\mu_d(A_0)=\frac{|A_0\cap F_dy_d|}{|F_d|}=0\not>\frac1n$. 
\end{proof}

Theorem~\ref{Lem1} implies that for an invariant ideal $\I$ in an amenable group $G$ the intersection 
$$\N_\I=\bigcap_{\N_\mu\supset\I}\N_\mu$$is a well-defined ideal that contains $\I$. In this definition $\mu$ runs over all Banach measures on $G$ such that $\I\subset\N_\mu$. The ideal $\N_\I$ will be called the {\em $\N$-completion} of the ideal $\I$. An invariant ideal $\I$ is defined to be {\em $\N$-complete} if $\I$ coincides with its $\N$-completion $\N_\I$. 

The $\N$-completion $\N_{\{\emptyset\}}=\bigcap_\mu\N_\mu$ of the smallest ideal $\I=\{\emptyset\}$ is denoted by $\N$ and called the {\em ideal of absolute null sets}. The ideal $\N$ is well-defined for each amenable group $G$.

The following properties of $\N$-complete ideals follows immediately from the definition.
 
\begin{proposition} Let $G$ be an amenable group.
\begin{enumerate}
\item For every Banach measure $\mu$ the ideal $\N_\mu$ is $\N$-complete.
\item For any $\N$-complete ideals $\I_\alpha$, $\alpha\in A$, in $G$ the intersection $\bigcap_{\alpha\in A}\I_\alpha$ is an $\N$-complete ideal.
\item For each invariant ideal $\I$ in $G$ the ideal $\N_\I$ is $\N$-complete.
\item $\N_I=\N$ for each invariant ideal $\I\subset\N$.
\end{enumerate}
\end{proposition}

The $\CS$- and $\N$-completions relate as follows.

\begin{theorem}\label{NS} For any invariant ideal $\I$ in an amenable group $G$ we get
$\I\subset\N_\I\subset\CS_\I$. In particular, $\N\subset\CS$.
\end{theorem}

\begin{proof} Assume conversely that some set $A\in\N_\I$ is not $\I$-small.
This means that there is a finite set $F\subset G$ such that the complement $G\setminus FA$ is not $\I$-large. Consequently, the family $\I\cup(G\setminus FA)$ generates an invariant ideal 
$$\J=\{J\subset G:J\subset I\cup E(G\setminus FA)\mbox{ for some $I\in\I$ and $E\in\F$}\}.$$
By Theorem~\ref{Lem1}, there is a Banach measure $\mu$ on $G$ such that $J\subset \N_\mu$. Then $\mu(G\setminus FA)=0$ and hence $\mu(FA)=1$ and $\mu(A)>0$.
Now we see that $\I\subset\N_\I\subset \N_\mu$ but $A\notin\N_\mu$, which contradicts the choice of $A\in\N_\I$.
\end{proof} 

\begin{corollary} Each $\CS$-complete ideal in an amenable group in $\N$-complete.
\end{corollary}

\section{Packing indices}

Let $\I$ be an invariant ideal of subsets of a group $G$.
To each subset $A$ of the group $G$ we can assign the {\em $\I$-packing index}
$$\I\mbox{-}\pack(A)=\sup\{|B|:B\subset G\mbox{ is such that $\{bA\}_{b\in B}$ is
disjoint modulo $\I$}\}.$$ An indexed family $\{A_b\}_{b\in B}$ of subsets of $G$ is
called {\em disjoint modulo} the ideal $\I$ if $A_b\cap A_\beta\in\I$ for any distinct
 indices $b,\beta\in B$.

If $\I=\{\emptyset\}$ is the trivial ideal on $G$, then we write $\pack(A)$ instead
of $\{\emptyset\}\mbox{-}\pack(A)$.

For example, the packing index $\pack(2\IZ)$ of the set $A=2\IZ$ of even numbers in the group $G=\IZ$ is equal to 2. The same equality $\I\mbox{-}\pack(2\IZ)=2$ holds for any ideal $\I$ on $\IZ$.

It should be mentioned that in the definition of the $\I$-packing index, the supremum cannot be replaced by the maximum: by \cite{BL2}, each infinite group $G$ contains a subset $A\subset G$ such that $\pack(A)\ge\aleph_0$ but no infinite set $B\subset G$ with disjoint $\{bA\}_{b\in B}$ exists.

To catch the difference between sup and max, for a subset $A\subset G$ let us consider a more informative  cardinal characteristic
$$
\I\mbox{-}\Pack(A)=\sup\big\{|B|^+:B\subset G\mbox{ such that $\{bA\}_{b\in B}$  is disjoint modulo $\I$}\}.
$$
It is clear that $\I\mbox{-}\pack(A)\le\I\mbox{-}\Pack(A)$ and $$\I\mbox{-}\pack(A)=\sup\{\kappa:\kappa<\I\mbox{-}\Pack(A)\},$$
so the value of $\I\mbox{-}\pack(A)$ can be recovered from that of $\I\mbox{-}\Pack(A)$.
The packing indices $\{\emptyset\}\mbox{-}\pack$ and $\{\emptyset\}\mbox{-}\Pack$ were
intensively studied in \cite{BL1}, \cite{BL2}, \cite{BLR}, \cite{Las}.

In fact, the $\I$-packing indices $\I\mbox{-}\pack(A)$ and $\I\mbox{-}\Pack(A)$ are partial cases of the packing indices $\I\mbox{-}\pack_n(A)$ and $\I\mbox{-}\Pack_n(A)$ defined for every cardinal number $n\ge 2$ by the formulas:
$$\I\mbox{-}\pack_n(A)=\sup\big\{|B|:B\subset G\mbox{ such that }\{\bigcap_{c\in C}cA:C\in[B]^n\}\subset\I\big\}$$and
$$
\begin{aligned}
\I\mbox{-}\Pack_n(A)&=\sup\big\{|B|^+:B\subset G\mbox{ such that }\{\bigcap_{c\in C}cA:C\in[B]^n\}\subset\I\big\}=\\
&=\min\big\{\kappa:\forall B\subset G\;\big(|B|\ge\kappa\;\Ra\;(\exists C\in[B]^n\mbox{ with }\bigcap_{c\in C}cA\notin\I)\big)\big\},
\end{aligned}
$$ where $[B]^n$ stands for the family of all $n$-element subsets of $B$.

It is clear that $\I\mbox{-}\pack(A)=\I\mbox{-}\pack_2(A)$ and $\I\mbox{-}\Pack(A)=\I\mbox{-}\Pack_2(A)$.
Also, $\I\mbox{-}\pack_n(A)\le\I\mbox{-}\pack_{n+1}(A)$ and $\I\mbox{-}\Pack_n(A)\le\I\mbox{-}\Pack_{n+1}(A)$ for any finite $n\ge 2$.

If $\I=\{\emptyset\}$, then we shall write $\pack_n(A)$ and $\Pack_n(A)$ instead of $\{\emptyset\}\mbox{-}\pack_n(A)$ and $\{\emptyset\}\mbox{-}\Pack_n(A)$.

The following example show that the difference between the packing indices $\pack_2(A)$ and $\pack_3(A)$ can be infinite.

\begin{example} The subset $A=\{n(n-1)/2:n\in\IN\}$ of the group $\IZ$ has $\pack_2(A)=1$ and $\pack_3(A)=\aleph_0$. The latter equality follows from the observation that the family $\{2^m+A\}_{m\in\IN}$ is 3-disjoint in the sense that $(2^n+A)\cap(2^m+A)\cap(2^k+A)=\emptyset$ for any pairwise distinct numbers $n,m,k\in\IN$.
\end{example}

\section{Packing-complete ideals}

It is clear that for each set $A\in\I$ and finite $n\ge 2$ we get $\I\mbox{-}\pack_n(A)=|G|$ and $\I\mbox{-}\Pack_n(A)=|G|^+$. We shall be interested in ideals for which the converse implication holds.

\begin{definition} An invariant ideal $\I$ on a group $G$ is called {\em $\Pack_n$-complete} (resp. {\em $\pack_n$-complete}) if
$\I$ contains each set $A\subset G$ with $\I\mbox{-}\Pack_n(A)\ge\aleph_0$ (resp. $\I\mbox{-}\pack_n(A)\ge\aleph_0$).
\end{definition}

Since $\Pack_n(A)\le\Pack_{n+1}(A)$, each $\Pack_{n+1}$-complete ideal is $\Pack_n$-complete.

\begin{definition} An invariant ideal $\I$ on a group $G$ is called {\em $\Pack_{<\w}$-complete} if $\I$ is $\Pack_n$-complete for every $n\ge 2$.
\end{definition}

For each ideal $\I$ in a group $G$ and every $n\ge 2$ we get the implications:
$$\mbox{$\Pack_{<\w}$-complete \ $\Ra$ \ $\Pack_n$-complete \ $\Ra$ \ $\Pack_2$-complete \ $\Ra$ \ $\pack_2$-complete}.$$

The simplest example of a $\Pack_{<\w}$-complete ideal is the ideal $\N_\mu$.

\begin{theorem}\label{Nmu} For any Banach measure $\mu$ on a group $G$  the ideal $\N_\mu=\{A\subset G:\mu(A)=0\}$ is $\Pack_{<\w}$-complete.
\end{theorem}

 \begin{proof} We need to show that the ideal $\N_\mu$ is $\Pack_n$-complete for every $n\ge 2$.
 Take any subset $A\subset G$ with
$\mathcal{N}_\mu \mbox{-}\Pack_n(A)\ge\aleph_0$.

This means that for any natural number $m$ there is a subset $B_m$
of size $m$ such that $\mu(b_1A\cap \dots \cap b_nA)=0$
for any distinct $b_1,\dots b_n\in B_m$. It follows that $$1\ge\mu
(\bigcup_{b\in B_m} bA)\ge \frac{1}{n}\sum_{b\in B_m} \mu
(bA)=\frac{1}{n}\sum_{b\in B_m}\mu (A)$$and hence $\mu(A)\leq
\frac{n}{|B_m|}=\frac{n}{m}$. Since this equality holds for any
$m$ we conclude that $\mu(A)=0$ and hence $A\in\N_\mu$.
\end{proof}

Since the intersection of $\Pack_{<\w}$-complete ideals is $\Pack_{<\w}$-complete, Theorems~\ref{Nmu} and \ref{NS} imply

\begin{corollary}\label{PackNI} An invariant ideal $\I$ in an amenable group $G$ is $\Pack_{<\w}$-complete provided $\I$ is $\N$-complete or $\CS$-complete. 
\end{corollary}

Thus for each ideal $\I$ in an amenable group $G$ and every $n\ge 2$ we get the implications:
$$\mbox{$\CS$-complete \ $\Ra$ \ $\N$-complete \ $\Ra$ \ $\Pack_{<\w}$-complete \  \ $\Ra$ \ $\pack_2$-complete}.$$

The amenability assumption is essential in Corollary~\ref{PackNI}.

\begin{proposition}\label{F2} If a group $G$ contains an isomorphic copy of the free group $F_2$ with two generators, then $G=A\cup B$ is the union of two subsets with infinite $\pack_2$-index. Consequently, no ideal of $G$ is $\pack_2$-complete. In particular, the ideal $\CS$ of small subsets of $G$ is not $\pack_2$-complete.
\end{proposition} 

\begin{proof} Let $a,b$ be the generators of the free subgroup $F_2\subset G$. Choose a subset $S\subset G$ that meets each coset $F_2\cdot g$, $g\in G$, at a single point. We shall additionally assume that the singleton $S\cap F_2=\{e\}$ contains the neutral element of $G$. Each element of $F_2$ can be uniquely written as an irreducible word in the alphabet $\{a,a^{-1},b,b^{-1}\}$. Let $F_a$ be the set of irreducible words that start with $a$ or $a^{-1}$. It is clear that $F_2=F_a\cup(F_2\setminus F_a)$ and thus $G=A\cup B$ where $A=F_aS$ and $B=G\setminus A=(F_2\setminus F_a)S$. It remains to observe that the sets $A$, $B$ have infinite $\pack_2$-index. 

Assuming that $G$ contains a $\pack_2$-complete ideal $\I$, we conclude that $A,B\in\I$ and hence $G=A\cup B\in\I$, which is a contradiction.
\end{proof}

\section{$\Pack_{<\w}$-completion of ideals}

For an ideal $\I$ in a group $G$ let $\Pack_{<\w}(\I)$ be the intersection of all invariant $\Pack_{<\w}$-complete ideals $\J\subset\mathcal P(G)$ that contain $\I$.
If no such an ideal $\J$ exists, then we put $\Pack_{<\w}(\I)=\mathcal P(G)$. The family $\Pack_{<\w}(\I)$ is called the $\Pack_{<\w}$-completion on $\I$. If $\Pack_{<\w}(\I)\ne\mathcal P(G)$, then $\Pack_{<\w}(\I)$ is a $\Pack_{<\w}$-complete ideal in $G$. 

By analogy, for every $n\ge 2$ we can define the $\Pack_n$-completion $\Pack_n(\I)$ of $\I$. Corollary~\ref{PackNI} guarantees that for an invariant ideal $\I$ in an amenable group $G$ its packing completions $\Pack_{<\w}(\I)$ and $\Pack_n(\I)$ are ideals lying in the $\Pack_{<\w}$-complete ideal $\N_\I$. On the other hand, for a group $G$ containing a copy of the free group $F_2$, the packing completions $\Pack_n(\I)$ and $\Pack_{<\w}(\I)$ coincide with $\mathcal P(G)$.

The following theorem describes the inner structure of the $\Pack_{<\w}$-completion.
A subset $\mathcal A\subset\mathcal P(G)$ is called {\em additive} if $A\cup B$ for any sets $A,B\in\mathcal A$.

\begin{proposition}\label{t5} The $\Pack_{<\w}$-completion $\Pack_{<\w}(\I)$ of an invariant ideal $\I$ on a group $G$ is equal to the union
$$\mathcal{I}_{<\omega_1}=\bigcup_{\alpha<\omega_1}\mathcal{I}_\alpha$$
 where $\mathcal{I}_0 = \mathcal{I}$ and  $\mathcal{I}_\alpha$ is
 the smallest additive family containing all subsets $A\subset G$ with infinite index
 $\mathcal{I}_\beta$-$\Pack_n(A)$  for some $\beta<\alpha$ and $n<\w$.
\end{proposition}

\begin{proof} 
The inclusion  $\Pack_{<\w}(\I)\supset \mathcal{I}_{<\omega_1}$
follows from the fact that each $\Pack_{<\omega}$-complete ideal
which contains $\mathcal{I_\beta}$ for all $\beta <\alpha$ also
contains $\mathcal{I}_\alpha$. To show the equality we need to
prove that $\mathcal{I}_{<\omega_1}$ is a $\Pack_{<\omega}$-complete
ideal if $\mathcal{I}_{<\omega_1}\ne\mathcal P(G)$.

First we show that $\mathcal{I}_{<\omega_1}$ is $\Pack_n$-complete
for each $n\ge 2$. Let $A$ be subset of $G$ with
$\mathcal{I}_{<\omega_1}\mbox{-}\Pack_n(A)\ge \aleph_0$. It means
that there is countable sequence $(B_m)_{m\in \omega}$ of subsets
$B_m\in [G]^m$ such that for any subset $C\in [B_m]^n$ the
intersection $\bigcap_{c\in C} cA\in \mathcal{I}_{<\omega_1}$ and
thus $\bigcap_{c\in C} cA\in \mathcal{I}_{\alpha(m,C)}$ for some
countable ordinal $\alpha(m,C)<\w_1$. Since the sequence $(B_m)$
is countable and the sets $B_m$ and $C\in [B_m]^n$ are finite,
there is a countable ordinal $\alpha$ such that $\alpha
>\alpha(m,C)$ for each $m$ and each $C$. For this ordinal $\alpha$
we get $\mathcal{I}_\alpha\mbox{-}\Pack_n(A)\ge \aleph_0$.
According to the definition of $\mathcal{I}_{\alpha+1}$ the set
$A\in \mathcal{I}_{\alpha+1}\subset \mathcal{I}_{<\omega}$ and
thus $\mathcal{I}_{<\omega_1}$ is $\Pack_n$-complete for each
$n\ge 2$.
\end{proof}

By analogy we can describe the $\Pack_n$-completion $\Pack_n(\I)$ of $\I$.

\begin{proposition} The $\Pack_{n}$-completion $\Pack_{n}$ of an invariant ideal $\I$ on a group $G$ is equal to the union
$$\mathcal{I}_{<\omega_1}=\bigcup_{\alpha<\omega_1}\mathcal{I}_\alpha$$
 where $\mathcal{I}_0 = \mathcal{I}$ and  $\mathcal{I}_\alpha$ is the smallest additive family that contains all subsets $A\subset G$ with infinite index
 $\mathcal{I}_\beta$-$\Pack_n(A)$ for some $\beta<\alpha$.
\end{proposition}

\section{Some Open problems}

For an invariant ideal $\I$ in an amenable group $G$ and every $n\ge 2$ we get the following chain of ideals:
$$\I\subset\Pack_2(\I)\subset\dots\subset\Pack_n(\I)\subset\Pack_{n+1}(\I)\subset\dots\subset\Pack_{<\w}(\I)\subset\N_\I\subset\CS_\I.$$
For the ideal $\I=\CS$ all these ideals coincide. What happens for the smallest ideal $\I_0=\{\emptyset\}$?

\begin{problem} Are the ideals $\Pack_n(\I_0)$, $n\ge2$, in the group $\IZ$  pairwise distinct? Do they differ from the ideal $\Pack_{<\w}(\I_0)$? Is $\Pack_{<\w}(\I_0)\ne\N$?
\end{problem}

It should be mentioned that $\N\ne\CS$ for each countable amenable group, see \cite[5.3]{LP}.

\begin{problem} Give a combinatorial description of subsets lying in the ideals $\Pack_{<\w}(\I_0)$ and $\Pack_n(\I_0)$, $n\ge 2$.
\end{problem}

Each ideal of subsets of the group $\IZ$ can be considered as a subspace of the Cantor cube $\{0,1\}^\IZ$. So, we can speak about topological properties of ideals.

\begin{problem} Describe the Borel complexity of the ideals $\Pack_{<\w}(\I_0)$ and $\Pack_n(\I_0)$ for $n\ge 2$. Are they coanalytic? Is the ideal $\N$ of absolute null sets coanalytic?
\end{problem}

Let us recall that a subspace $C$ of a Polish space $X$ is {\em coanalytic} if its complement $X\setminus C$ is analytic. The latter means that $X\setminus C$ is the continuous image of a Polish space, see \cite{Ke}. It easy to show that the ideal $\CS$ of small subsets of a countable group $G$ is an $F_{\sigma\delta}$-subset of $\mathcal P(G)$.
\smallskip

By Corollary~\ref{PackNI}, for the smallest ideal $\I_0=\{\emptyset\}$ in an amenable group $G$ its packing completion $\Pack_{<\w}(\I_0)\subset\CS$.  On the other hand, Proposition~\ref{F2} implies that $\Pack_{<\w}(\I_0)=\mathcal P(G)$ if $G$ contains a copy of the free group $F_2$.

\begin{problem}[I.Protasov] Is a group $G$ amenable if $\Pack_{<\w}(\I_0)\subset\CS$?
\end{problem}

\section{Acknowledgment}

The authors would like to thank Igor Protasov for careful and creative reading the manuscript and many valuable remarks and suggestions that substantially improved  both the results and presentation of the paper.

\end{document}